\documentclass[a4paper,12pt]{article}
\usepackage[hiresbb]{graphicx}
\usepackage{amsmath,amsthm,amssymb}
\usepackage{natbib,color,bm}
\numberwithin{equation}{section}
\usepackage{setspace} 
\usepackage{diagbox}
\usepackage[top=30truemm,bottom=30truemm,left=25truemm,right=25truemm]{geometry}
\usepackage{algorithmicx}%
\usepackage{algpseudocode}%
\usepackage{listings}%
\usepackage{anyfontsize}%
\usepackage{diagbox}
\usepackage{comment}
\usepackage{booktabs}
\usepackage{authblk}

\usepackage{hyperref}
\hypersetup{%
  colorlinks=true,
  linkcolor=blue,  
  citecolor=black,   
  urlcolor=blue     
}

\makeatletter
\AtBeginDocument{%
  \@ifundefined{NAT@hyper@}{}{%
    \let\NAT@orig@hyper@\NAT@hyper@
    \renewcommand*{\NAT@hyper@}[1]{%
      \begingroup
        \let\NAT@orig@date\NAT@date
        \def\NAT@date{\textcolor{blue}{\NAT@orig@date}}%
        \NAT@orig@hyper@{#1}%
      \endgroup
    }%
  }%
}
\makeatother
 
\theoremstyle{plain} 
\newtheorem{thm}{Theorem}[section]%
\newtheorem{lem}[thm]{Lemma}%
\newtheorem{cor}[thm]{Corollary}%

\theoremstyle{definition}

\newtheorem{rem}{Remark}%

\numberwithin{equation}{section}

\def\MC{ {\mathcal{C}} }
\def\MN{ {\mathcal{N}} }
\def\MS{ {\mathcal{S}} }
\def\tr{ {\mathrm{tr}}}
\def\ve{ \varepsilon }
\def\rk{ {\mathrm{rank}} }
\def\E{ {\mathrm{E}}}
\def\V{ {\mathrm{Cov}}}
\def\fa{ {\rm for \ all} }

\def\an{ {\rm and} }
\def\diag{ {\rm diag} }
\def\ftd{ {\rm for \ two \ different} }

\title{\Large\textbf{Parameter-free conditions for equality of general ridge estimators under spatial error models}}

\author[1,2]{Hirai Mukasa}

\affil[1]{Graduate School of Mathematics, Kyushu University}
\affil[2]{National Institute of Technology, Kagoshima College}

\date{}

\begin{document}
\maketitle

\begin{abstract}
This paper investigates when two general ridge estimators coincide under spatial error models. 
First, in the general linear model, we derive a necessary and sufficient condition based on the commutativity of an extended dispersion matrix and an orthogonal projector, thereby extending the classical result of Zyskind. 
Next, we establish parameter-free conditions for first-order spatial autoregressive and spatial moving average processes. 
Also, we obtain a necessary and sufficient condition valid for all values of the spatial correlation coefficient in the specified parameter range and characterize all penalty matrices for which the two general ridge estimators coincide.
A numerical experiment illustrates the theoretical results and shows that the proposed conditions can simplify the two-step estimation procedure by eliminating the need to estimate the spatial correlation coefficient.
\end{abstract}

\section{Introduction}\label{sec:1}
For positive integers $m_1$ and $m_2$, let $\mathbb{R}^{m_1 \times m_2}$ denote the set of $m_1 \times m_2$ real matrices.
For a positive integer $m$, let $\MS^+(m)$ and $\MS^N(m)$ denote the sets of $m \times m$ positive definite and positive semidefinite matrices, respectively. 
Furthermore, let $\boldsymbol{I}_m \in \MS^+(m)$ denote the $m \times m$ identity matrix.
The transpose, rank, trace, column space, and null space of a matrix $\boldsymbol{A}$ are denoted by $\boldsymbol{A}^\top$, $\rk(\boldsymbol{A})$, $\tr(\boldsymbol{A})$, $\MC(\boldsymbol{A})$, and $\MN(\boldsymbol{A})$, respectively.
$\|\cdot\|_{\mathrm{F}}$ stands for the Frobenius norm.
For matrices $\boldsymbol{A}, \boldsymbol{B} \in \mathbb{R}^{m_1 \times m_2}$, the Frobenius inner product is defined by $\left\langle \boldsymbol{A},\boldsymbol{B}\right\rangle_{\mathrm{F}}
= \tr(\boldsymbol{A}^{\top}\boldsymbol{B})$.
Let $\otimes$ denote the Kronecker product.
The expectation vector and covariance matrix of a random vector $\boldsymbol{v}$ are denoted by $\E[\boldsymbol{v}]$ and $\V(\boldsymbol{v})$, respectively.

Consider the general linear model
\begin{align}\label{glm}
 \boldsymbol{y} = \boldsymbol{X} \boldsymbol{\beta} +\boldsymbol{u}, \quad \E[\boldsymbol{u}] = \boldsymbol{0}, \quad \V(\boldsymbol{u}) = \sigma^2 \boldsymbol{\Omega},  
\end{align}
where $\boldsymbol{y}\in \mathbb{R}^n$ is an observable random vector, $\boldsymbol{X} \in \mathbb{R}^{n \times k}$ is a design matrix $(n>k)$ with rank $k$, $\boldsymbol{\beta} \in \mathbb{R}^k$ is an unknown regression coefficient vector, $\sigma^2 >0$ is an unknown scalar, and $\boldsymbol{\Omega} \in \MS^+(n)$ is a dispersion matrix.
The case where $\boldsymbol{\Omega}\neq\boldsymbol{I}_n$ is of 
interest because it allows for nonspherical covariance structures in the disturbances.
In particular, spatial error models have been
widely studied in spatial econometrics and applied to regional economic and
social phenomena.
For example, \cite{RefV24} employed a spatial error panel model
to examine the relationship between income inequality and household debt
across U.S. states, while \cite{RefK24} applied this
model to analyse the political gender gap across Japanese prefectures.
Recent methodological developments include causal inference with spatial
econometric models in epidemiologic research and robustness analysis of
spatial econometric estimators; see \cite{RefC25} and
\cite{RefNS25}.
Moreover, \cite{RefTE24} considered a spatial autoregressive panel data model in which the disturbances follow a spatial moving average process.
In light of these developments, we focus on spatial error models.

To estimate $\boldsymbol{\beta}$ in the general linear model, the generalized least squares estimator (GLSE) and the ordinary least squares estimator (OLSE) have been widely used as classical unbiased estimators.
However, when the explanatory variables are highly correlated, these estimators may have large variances, resulting in unstable estimates. 
A representative method for alleviating this problem is ridge estimation, proposed by \cite{RefHK70}. 
More generally, for $\boldsymbol{\Phi}\in\MS^+(n)$ and $\boldsymbol{K}\in\MS^N(k)$, we consider the following regularization problem:
\begin{align}\label{OP}
\min_{\boldsymbol{b}\in\mathbb{R}^{k}}
\left\{
(\boldsymbol{y}-\boldsymbol{X}\boldsymbol{b})^{\top}
\boldsymbol{\Phi}^{-1}
(\boldsymbol{y}-\boldsymbol{X}\boldsymbol{b})
+
\boldsymbol{b}^{\top}\boldsymbol{K}\boldsymbol{b}
\right\}.
\end{align}
The unique solution $\boldsymbol{b}_*$ of \eqref{OP} is
\begin{align}\label{GRE}
\hat{\boldsymbol{\beta}}_{GR}(\boldsymbol{\Phi},\boldsymbol{K}) = (\boldsymbol{X}^\top \boldsymbol{\Phi}^{-1}\boldsymbol{X} + \boldsymbol{K})^{-1} \boldsymbol{X}^\top \boldsymbol{\Phi}^{-1}\boldsymbol{y},
\end{align}
which is the so-called general ridge estimator \citep{RefR75}.

Hereafter, we consider the case where $\boldsymbol{\Phi} \in \{\boldsymbol{\Omega}, \boldsymbol{I}_n\}$ and $\boldsymbol{K}\in\MS^N(k)$ is fixed.
The class of general ridge estimators includes several well-known estimators as special cases.
For instance, $\hat{\boldsymbol{\beta}}_{GR}(\boldsymbol{\Omega},\boldsymbol{0})$ coincides with GLSE, whereas $\hat{\boldsymbol{\beta}}_{GR}(\boldsymbol{I}_n,\boldsymbol{0})$ coincides with OLSE.
Moreover, $\hat{\boldsymbol{\beta}}_{GR}(\boldsymbol{I}_n, \lambda\boldsymbol{I}_k)$ reduces to the ordinary ridge estimator, where $\lambda$ is a positive constant.
When $\boldsymbol{\Omega}$ is known, $\hat{\boldsymbol{\beta}}_{GR}(\boldsymbol{\Omega},\boldsymbol{K})$ possesses several desirable properties such as linear sufficiency and admissibility within the class of linear estimators under the squared loss \citep{RefM96}.
When $\boldsymbol{K} \in \MS^+(k)$, $\hat{\boldsymbol{\beta}}_{GR}(\boldsymbol{\Omega}, \boldsymbol{K})$ is obtained as an optimal estimator under a minimax criterion for the risk function associated with squared loss.
For these properties, see, for example, \cite{RefR76, RefD83, RefBM88, RefAS00, RefAS11, RefM26}.
In particular, $\hat{\boldsymbol{\beta}}_{GR}(\boldsymbol{\Omega}, \boldsymbol{0})$ is the best linear unbiased estimator for $\boldsymbol{\beta}$ in the general linear model.
Hence, it is natural to use $\hat{\boldsymbol{\beta}}_{GR}(\boldsymbol{\Omega},\boldsymbol{K})$ as typical linear estimators when $\boldsymbol{\Omega}$ is known.

However, $\boldsymbol{\Omega}$ usually involves unknown parameters, such as a spatial correlation coefficient $\rho$.
This means that $\hat{\boldsymbol{\beta}}_{GR}(\boldsymbol{\Omega},\boldsymbol{K})$ cannot be directly used in practice.
So we adopt the following two-step estimation procedure.\vspace{1ex}\\
\vspace{1ex}
\textbf{Step 1.} Estimate $\rho$ by $\hat{\rho}$ and construct $\hat{\boldsymbol{\Omega}} = \boldsymbol{\Omega}(\hat{\rho})$.\\
\vspace{1ex}
\textbf{Step 2.} Compute $\hat{\boldsymbol{\beta}}_{GR}(\hat{\boldsymbol{\Omega}},\boldsymbol{K})$.\\
The two-step estimation procedure may be affected by the estimation error in $\hat{\rho}$.
To eliminate the need for this process, it is natural to ask under what conditions $\hat{\boldsymbol{\beta}}_{GR} (\boldsymbol{\Omega},\boldsymbol{K})$ coincides with $\hat{\boldsymbol{\beta}}_{GR} (\boldsymbol{I}_n,\boldsymbol{K})$.
In particular, when $\boldsymbol{K} = \boldsymbol{0}$, this problem reduces to the equality between OLSE and GLSE. 
\cite{RefK68} characterized  this equality in terms of the invariance of $\MC(\boldsymbol{X})$ under $\boldsymbol{\Omega}$, whereas \cite{RefZ67} established an equivalent condition based on the commutativity of $\boldsymbol{\Omega}$ and the orthogonal projector onto $\MC(\boldsymbol{X})$. For spatial error models, \cite{RefG01} derived parameter-free conditions under which OLSE coincides with GLSE.
Criteria for determining the equality of general ridge estimators in the general linear model have been proposed by \cite{RefTK20} and \cite{RefMT25}.
In addition to these existing results, we derive a new necessary and sufficient condition under which two general ridge estimators coincide based on the result of \cite{RefZ67}.
However, these general conditions involve unknown parameters and are not directly applicable in practice.
Therefore, this paper also derives parameter-free conditions for verifying the equality of general ridge estimators under spatial error models.

This paper is organized as follows.
Section~\ref{sec:2} introduces spatial error models considered in this paper. 
Section~\ref{sec:3} derives the main results on the equality of general ridge estimators. Section~\ref{sec:4} presents a numerical experiment to illustrate the theoretical findings. 
Finally, some concluding remarks are presented in Section~\ref{sec:5}.
All proofs are provided in the Appendix.

\section{Spatial error models}\label{sec:2}
In this section, we introduce the spatial error models with first-order spatial autoregressive (SAR(1)) and spatial moving average (SMA(1)) processes. 
Then we derive the corresponding dispersion matrices and discuss the admissible values of the spatial correlation coefficient.

Let the study area be partitioned into $n$ nonoverlapping regions $R_1, \ldots, R_n$.
Spatial dependence among these regions is represented by a known weights matrix $\boldsymbol{W} = (w_{ij})\in \mathbb{R}^{n \times n}$, where $w_{ij} \geq 0$ measures the spatial influence of $R_j$ on $R_i$.
It is usually assumed that $w_{ii} = 0$ for all $i = 1,\ldots, n$.
The weights matrix is often row-standardized so that
\begin{align*}
\sum_{j=1}^n w_{ij} = 1, \quad i = 1, \ldots, n.
\end{align*}
Note that $\boldsymbol{W}$ need not be symmetric, since the spatial influence of $R_j$ on $R_i$ may differ from that of $R_i$ on $R_j$.\\
\underline{\textbf{SAR(1) process}}\\
Suppose that the disturbance vector $\boldsymbol{u}$ in \eqref{glm} follows SAR(1) process
\begin{align}\label{SAR}
\boldsymbol{u} = \rho \boldsymbol{W}\boldsymbol{u} + \boldsymbol{\ve},
\end{align}
where $\rho \in \mathbb{R}$ is an unknown spatial correlation coefficient and $\boldsymbol{\ve} \in \mathbb{R}^n$ is a random vector such that $\E[\boldsymbol{\ve}] = \boldsymbol{0}$ and $\V(\boldsymbol{\ve}) = \sigma^2 \boldsymbol{I}_n$.
If $\boldsymbol{I}_n - \rho \boldsymbol{W}$ is nonsingular, then \eqref{SAR} can be written as $\boldsymbol{u} = (\boldsymbol{I}_n - \rho \boldsymbol{W})^{-1}\boldsymbol{\ve}$.
Thus, $\boldsymbol{\Omega}$ can be expressed as 
\begin{align*}
\boldsymbol{\Omega} = (\boldsymbol{I}_n - \rho \boldsymbol{W})^{-1}(\boldsymbol{I}_n - \rho \boldsymbol{W}^\top)^{-1}.
\end{align*}
Hereafter, we discuss the nonsingularity of $\boldsymbol{\Omega}$.
If there exists a matrix norm $\|\cdot\|$ such that $|\rho| \, \|\boldsymbol{W}\| < 1$, then $\boldsymbol{\Omega}$ is nonsingular.  
Each row of $\boldsymbol{W}$ sums to one,
so the maximum row sum matrix-norm is 
\begin{align*}
\|\boldsymbol{W}\|_{\infty} = \displaystyle \max_{1 \leq i \leq n} \sum_{j=1}^n |w_{ij}| = \displaystyle \max_{1 \leq i \leq n} \sum_{j=1}^n w_{ij} = 1.
\end{align*}
This implies that if $|\rho| < 1$, then $\boldsymbol{\Omega}$ is nonsingular.\\
\underline{\textbf{SMA(1) process}}\\
Similarly, we assume that $\boldsymbol{u}$ in \eqref{glm} follows SMA(1) process
\begin{align}\label{SMA}
\boldsymbol{u} = \rho \boldsymbol{W}\boldsymbol{\ve} + \boldsymbol{\ve},
\end{align}
where $\E[\boldsymbol{\ve}] = \boldsymbol{0}$ and $\V(\boldsymbol{\ve}) = \sigma^2 \boldsymbol{I}_n$.
Since \eqref{SMA} can equivalently be written as $\boldsymbol{u} = (\boldsymbol{I}_n + \rho \boldsymbol{W})\boldsymbol{\ve}$, it follows that 
\begin{align*}
\boldsymbol{\Omega} = (\boldsymbol{I}_n + \rho \boldsymbol{W})(\boldsymbol{I}_n + \rho \boldsymbol{W}^\top).
\end{align*}
By a similar argument, $|\rho| < 1$ is a sufficient condition for $\boldsymbol{\Omega}$ to be nonsingular.

Above all, $\boldsymbol{\Omega} = \boldsymbol{\Omega}(\rho)$ can be expressed as
\begin{align}\label{omega}
\boldsymbol{\Omega}(\rho) = \left\{ \,
    \begin{aligned}
    & (\boldsymbol{I}_n - \rho \boldsymbol{W})^{-1}(\boldsymbol{I}_n - \rho \boldsymbol{W}^\top)^{-1} \quad (\text{\textbf{SAR(1) process}})\\
    & (\boldsymbol{I}_n + \rho \boldsymbol{W})(\boldsymbol{I}_n + \rho \boldsymbol{W}^\top) \quad (\text{\textbf{SMA(1) process}})
    \end{aligned}
\right.,
\end{align}
and is nonsingular if $|\rho| < 1$.
\section{Equality of two general ridge estimators}\label{sec:3}
In this section, we first derive a necessary and sufficient condition for the equality
\begin{align}\label{G2E}
\hat{\boldsymbol{\beta}}_{GR}(\boldsymbol{\Omega}, \boldsymbol{K}) = \hat{\boldsymbol{\beta}}_{GR}(\boldsymbol{I}_n, \boldsymbol{K}) \quad \fa \ \ \boldsymbol{y} \in \mathbb{R}^n
\end{align}
to hold under the general linear model.
Then we specialize this condition to the spatial error models considered in Section~\ref{sec:2}.

\subsection{Equality of general ridge estimators under the general linear model}\label{sec:31}
\cite{RefMT25} derived the necessary and sufficient condition for the equality \eqref{G2E}:
\begin{align}\label{Mcond}
\MC\left(\begin{matrix}
\boldsymbol{\Omega X}\\
\boldsymbol{K}
\end{matrix}\right) = \MC\left(\begin{matrix}
\boldsymbol{X}\\
\boldsymbol{K}
\end{matrix}\right).
\end{align}
Since both matrices have full column rank $k$, \eqref{Mcond} holds if and only if there exists a nonsingular matrix $\boldsymbol{G} \in \mathbb{R}^{k \times k}$ such that $\boldsymbol{\Omega}\boldsymbol{X} = \boldsymbol{XG}$ and $\boldsymbol{K} = \boldsymbol{KG}$.
 Note that if $\boldsymbol{K} = \boldsymbol{0}$, \eqref{Mcond} becomes $\MC(\boldsymbol{\Omega}\boldsymbol{X}) = \MC(\boldsymbol{X})$, which is equivalent to the result of \cite{RefK68}.

To simplify the notation, let 
\begin{align*}
\widetilde{\boldsymbol{\Omega}} = 
\begin{pmatrix}
\boldsymbol{\Omega} & \boldsymbol{0}\\
\boldsymbol{0} & \boldsymbol{I}_k
\end{pmatrix}
\in \MS^+(n+k) \quad \an \quad \boldsymbol{Z} = 
\begin{pmatrix}
\boldsymbol{X}\\
\boldsymbol{K}
\end{pmatrix} \in \mathbb{R}^{(n+k) \times k}.
\end{align*}
Since $\boldsymbol{Z}$ has full column rank, let $\boldsymbol{P}_Z  = \boldsymbol{Z}(\boldsymbol{Z}^\top \boldsymbol{Z})^{-1}\boldsymbol{Z}^\top \in \mathbb{R}^{(n+k) \times (n+k)}$ denote the orthogonal projector onto $\MC(\boldsymbol{Z})$.
Then \eqref{Mcond} can be written as
\begin{align}\label{Mcond2}
\MC(\widetilde{\boldsymbol{\Omega}}\boldsymbol{Z}) = \MC(\boldsymbol{Z}).
\end{align}
\eqref{Mcond2} means that the column space of $\boldsymbol{Z}$ is invariant under $\widetilde{\boldsymbol{\Omega}}$.
Then we obtain the following equivalent condition by using the column space relationship.
\begin{thm}\label{thm1}
The equality \eqref{G2E} holds if and only if
\begin{align}\label{Zcond}
\widetilde{\boldsymbol{\Omega}}\boldsymbol{P}_Z = \boldsymbol{P}_Z\widetilde{\boldsymbol{\Omega}}.
\end{align}
\end{thm}

\begin{rem}
(i) Theorem~\ref{thm1} characterizes the equality of the two general ridge estimators in terms of the commutativity of $\widetilde{\boldsymbol{\Omega}}$ and $\boldsymbol{P}_Z$.\\
(ii) When $\boldsymbol{K} = \boldsymbol{0}$, we have 
\begin{align*}
\boldsymbol{Z} = 
\begin{pmatrix}
\boldsymbol{X}\\
\boldsymbol{0}
\end{pmatrix} \quad  \an \quad 
\boldsymbol{P}_Z = 
\begin{pmatrix}
\boldsymbol{P}_X & \boldsymbol{0}\\
\boldsymbol{0} & \boldsymbol{0}
\end{pmatrix},
\end{align*}
where $\boldsymbol{P}_X = \boldsymbol{X}(\boldsymbol{X}^\top\boldsymbol{X})^{-1}\boldsymbol{X}^\top$.
In this case, \eqref{Zcond} reduces to $\boldsymbol{\Omega}\boldsymbol{P}_X = \boldsymbol{P}_X\boldsymbol{\Omega}$.
Thus, \eqref{Zcond} can be viewed as an extension of Zyskind's condition to general ridge estimators.
\end{rem}

\subsection{Equality of general ridge estimators under spatial error models}\label{sec:32}
In Section~\ref{sec:31}, we characterized the equality of two general ridge estimators in terms of two conditions. 
However, under spatial error models, $\boldsymbol{\Omega}$ depends on the unknown parameter $\rho$, and hence we need more practical conditions to check their equality. 
In this subsection, we provide algebraically verifiable conditions that do not depend on $\rho$.

We apply condition \eqref{Mcond} to spatial error models as follows.
\begin{thm}\label{thm2}
Let $\boldsymbol{\Omega}$ be written in the form \eqref{omega} with $\rho \neq 0$.
If 
\begin{align}\label{new}
\MC\left(\begin{matrix}
\boldsymbol{WX}\\
\boldsymbol{0}
\end{matrix}\right) \subseteq \MC\left(\begin{matrix}
\boldsymbol{X}\\
\boldsymbol{K}
\end{matrix}\right) \quad \an \quad \MC\left(\begin{matrix}
\boldsymbol{W}^\top \boldsymbol{X}\\
\boldsymbol{0}
\end{matrix}\right) \subseteq \MC\left(\begin{matrix}
\boldsymbol{X}\\
\boldsymbol{K}
\end{matrix}\right),
\end{align}
then the equality \eqref{G2E} holds.
\end{thm}

\begin{rem}\label{rm5}
(a) When $\boldsymbol{K}=\boldsymbol{0}$, the condition \eqref{new} can be written as $\MC(\boldsymbol{WX}) \subseteq \MC(\boldsymbol{X})$ and $\MC(\boldsymbol{W}^\top \boldsymbol{X}) \subseteq \MC(\boldsymbol{X})$.
Thus, it reduces to the result of \cite{RefG01}.\\
(b) Suppose that $\boldsymbol{K} \in \MS^+(k)$.
If $\MC(\boldsymbol{X}) \subseteq \mathcal{N}(\boldsymbol{W})  \cap  \mathcal{N}(\boldsymbol{W}^\top)$, then the equality \eqref{G2E} holds.\\
(c) \eqref{new} is not a necessary condition for the equality \eqref{G2E} in general.
For example, let
\begin{align*}
\boldsymbol{W} = \begin{pmatrix}
0 & 1 & 0 & 0 & 0\\
0 & 0 & 1 & 0 & 0\\
0 & 0 & 0 & 1 & 0\\
0 & 0 & 0 & 0 & 1\\
1 & 0 & 0 & 0 & 0
\end{pmatrix}, \
\boldsymbol{X} = \begin{pmatrix}
1 & 0\\
\cos \theta & \sin \theta \\
\cos 2\theta & \sin2 \theta \\
\cos 3\theta & \sin 3\theta \\
\cos 4\theta & \sin 4\theta 
\end{pmatrix}, \ \rho = -2\cos\theta, \ \theta = \cfrac{2\pi}{5}.
\end{align*}
Assume in addition that $\boldsymbol{K}\in \MS^+(k)$, and
$\boldsymbol{\Omega}$ is of the form $\boldsymbol{\Omega} = (\boldsymbol{I}_n + \rho \boldsymbol{W})(\boldsymbol{I}_n + \rho \boldsymbol{W}^\top)$.
Then it follows from $\boldsymbol{W}^5 = \boldsymbol{I}_5$ that
each eigenvalue of $\boldsymbol{W}$ is $\exp\left(2 m\pi i/5\right)$ $(m = 0,\ldots,4)$, and one of the eigenvectors can be written as 
$\boldsymbol{v} = (
1, w, w^2, w^3, w^4)^\top$, where $w = \exp\left(2\pi i/5\right)$ with $i$ being the imaginary unit.
Since  $\boldsymbol{Wv} = w\boldsymbol{v}$ and $\boldsymbol{W}^\top \boldsymbol{v} = w^{-1}\boldsymbol{v}$, it holds that
\begin{align*}
\boldsymbol{\Omega v} = (1 + \rho w)(1 + \rho w^{-1})\boldsymbol{v} = (1 + 2 \rho \cos\theta + \rho^2)\boldsymbol{v} = \boldsymbol{v}.
\end{align*}
Therefore, we have $\boldsymbol{\Omega X} = \boldsymbol{X} $, and thus the equality \eqref{G2E} holds.
However, a straightforward calculation yields that $\boldsymbol{WX} \neq \boldsymbol{0}$, so \eqref{new} is not satisfied.
\end{rem}
Hereafter, let
\begin{align*}
\mathcal{P} = \{\rho \in \mathbb{R} \ | \ \rho \neq 0  \ \an \ |\rho|<1 \}.
\end{align*}
As mentioned in Remark~\ref{rm5}, the converse of Theorem~\ref{thm2} is not true.
Then we provide the necessary and sufficient conditions.
\begin{thm}\label{thm3}
For each $\rho \in \mathcal{P}$, let $\boldsymbol{\Omega} = (\boldsymbol{I}_n + \rho \boldsymbol{W})(\boldsymbol{I}_n + \rho \boldsymbol{W}^\top)$.
Then the following statements are equivalent:
\begin{itemize}
\item[\textup{(i)}] $\MC\left(\begin{matrix}
\boldsymbol{\Omega X}\\
\boldsymbol{K}
\end{matrix}\right) = \MC\left(\begin{matrix}
\boldsymbol{X}\\
\boldsymbol{K}
\end{matrix}\right) \quad \fa \ \ \rho \in \mathcal{P}$.
\item[\textup{(ii)}] $\MC\left(\begin{matrix}
\boldsymbol{\Omega X}\\
\boldsymbol{K}
\end{matrix}\right) = \MC\left(\begin{matrix}
\boldsymbol{X}\\
\boldsymbol{K}
\end{matrix}\right) \quad \ftd \ \ \rho_1, \rho_2 \in \mathcal{P}$.
\item[\textup{(iii)}] $\MC\left(\begin{matrix}
(\boldsymbol{W}+\boldsymbol{W}^\top)\boldsymbol{X}\\
\boldsymbol{0}
\end{matrix}\right) \subseteq \MC\left(\begin{matrix}
\boldsymbol{X}\\
\boldsymbol{K}
\end{matrix}\right) \quad  \an \quad \MC\left(\begin{matrix}
\boldsymbol{W}\boldsymbol{W}^\top\boldsymbol{X}\\
\boldsymbol{0}
\end{matrix}\right) \subseteq \MC\left(\begin{matrix}
\boldsymbol{X}\\
\boldsymbol{K}
\end{matrix}\right)$.
\end{itemize}
\end{thm}

\begin{rem}\label{rem33}
(a) When $\boldsymbol{\Omega}$ is of the form $\boldsymbol{\Omega} = (\boldsymbol{I}_n - \rho \boldsymbol{W})^{-1}(\boldsymbol{I}_n - \rho \boldsymbol{W}^\top)^{-1}$, the condition (iii) is replaced by
\begin{align*}
\MC\left(\begin{matrix}
(\boldsymbol{W}+\boldsymbol{W}^\top)\boldsymbol{X}\\
\boldsymbol{0}
\end{matrix}\right) \subseteq \MC\left(\begin{matrix}
\boldsymbol{X}\\
\boldsymbol{K}
\end{matrix}\right) \quad  \an \quad \MC\left(\begin{matrix}
\boldsymbol{W}^\top\boldsymbol{W}\boldsymbol{X}\\
\boldsymbol{0}
\end{matrix}\right) \subseteq \MC\left(\begin{matrix}
\boldsymbol{X}\\
\boldsymbol{K}
\end{matrix}\right).
\end{align*}
(b) The condition (iii) is also equivalent to 
\begin{align*}
\boldsymbol{K}(\boldsymbol{X}^\top \boldsymbol{X})^{-1}\boldsymbol{X}^\top (\boldsymbol{W} + \boldsymbol{W}^\top) \boldsymbol{X} = \boldsymbol{K}(\boldsymbol{X}^\top \boldsymbol{X})^{-1}\boldsymbol{X}^\top \boldsymbol{W} \boldsymbol{W}^\top \boldsymbol{X}   = \boldsymbol{0}
\end{align*}
and 
\begin{align*}
(\boldsymbol{X}^\bot)^\top (\boldsymbol{W} + \boldsymbol{W}^\top) \boldsymbol{X}
= (\boldsymbol{X}^\bot)^\top \boldsymbol{W} \boldsymbol{W}^\top \boldsymbol{X} = \boldsymbol{0},
\end{align*}
where $\boldsymbol{X}^\bot \in \mathbb{R}^{n \times (n-k)}$ is any matrix such that $\boldsymbol{X}^\top \boldsymbol{X}^\bot = \boldsymbol{0}$ and $\rk( \boldsymbol{X}^\bot) = n-k$.\\
(c) Condition (iii) enables us to check whether the equality \eqref{G2E} holds without specifying the value of $\rho$.\\
(d) When $\boldsymbol{K}\in \MS^+(k)$, the condition (iii) is equivalent to 
\begin{align*}
\MC(\boldsymbol{X}) \subseteq \mathcal{N}(\boldsymbol{W}+\boldsymbol{W}^\top)  \cap  \mathcal{N}(\boldsymbol{W}\boldsymbol{W}^\top).    
\end{align*}
\end{rem}

From Theorem~\ref{thm3}, we obtain the following characterization of all matrices $\boldsymbol{K} \in \MS^N(k)$ for which the equality \eqref{G2E} holds for all $\boldsymbol{\rho} \in \mathcal{P}$.

\begin{cor}\label{cor1}
For each $\rho \in \mathcal{P}$, let $\boldsymbol{\Omega} = (\boldsymbol{I}_n + \rho \boldsymbol{W})(\boldsymbol{I}_n + \rho \boldsymbol{W}^\top)$.
Suppose that $\MC((\boldsymbol{W} + \boldsymbol{W}^\top)\boldsymbol{X}) \subseteq \MC(\boldsymbol{X})$ and $\MC(\boldsymbol{W}\boldsymbol{W}^\top\boldsymbol{X}) \subseteq \MC(\boldsymbol{X})$.
Define 
\begin{align*}
& \boldsymbol{R}_1 = (\boldsymbol{X}^\top \boldsymbol{X})^{-1}\boldsymbol{X}^\top(\boldsymbol{W} + \boldsymbol{W}^\top)\boldsymbol{X},\\
& \boldsymbol{R}_2 = (\boldsymbol{X}^\top \boldsymbol{X})^{-1}\boldsymbol{X}^\top\boldsymbol{W}\boldsymbol{W}^\top\boldsymbol{X},\\
& \boldsymbol{R} = (\boldsymbol{R}_1 : \boldsymbol{R}_2), \quad q = k - \rk(\boldsymbol{R}),
\end{align*}
and choose $\boldsymbol{Q} \in \mathbb{R}^{k \times q}$ whose columns form an orthonormal basis for $\MN(\boldsymbol{R}^\top)$.
Then the statement (iii) in Theorem~\ref{thm3} is equivalent to $\boldsymbol{K} = \boldsymbol{Q}\boldsymbol{H}\boldsymbol{Q}^\top$ for some $\boldsymbol{H} \in \MS^N(q)$.
\end{cor}

\begin{rem}
(a) The matrix $\boldsymbol{Q}$ in Corollary~\ref{cor1} depends only on $\boldsymbol{X}$ and $\boldsymbol{W}$, and hence the characterization of $\boldsymbol{K} \in \MS^N(k)$ is independent of $\rho$.\\
(b) It holds that $\rk(\boldsymbol{K}) = \rk(\boldsymbol{H}) \leq q$, and every rank from $0$ to $q$ can be attained by an appropriate choice of $\boldsymbol{H} \in \MS^N(q)$.
If $q = 0$, then $\boldsymbol{K} = \boldsymbol{0}$ is the only matrix for which the equality \eqref{G2E} holds.
If $0<q<k$, there exist nonzero admissible matrices, but they are necessarily singular.
If $q = k$, or equivalently $\boldsymbol{R} = \boldsymbol{0}$, then the equality \eqref{G2E} holds for every $\boldsymbol{K} \in \MS^N(k)$.
In particular, there exists a matrix $\boldsymbol{K} \in \MS^+(k)$ if and only if $q = k$.\\
(c) The class of matrices characterized in Corollary~\ref{cor1} can be written as
\begin{align*}
\mathcal{K} = \left\{\boldsymbol{K} \in \MS^N(k) \ | \ \boldsymbol{KR} = \boldsymbol{0}\right\}.
\end{align*}
This set forms a closed convex cone. Moreover, since
\begin{align*}
\boldsymbol{KR} = \boldsymbol{0} \quad \Leftrightarrow \quad \left\langle \boldsymbol{K},\boldsymbol{R}\boldsymbol{R}^\top\right\rangle_{\mathrm{F}} = \boldsymbol{0},
\end{align*}
the set $\mathcal{K}$ is an exposed face of the cone $\MS^N(k)$.
\end{rem}

For simplicity, let us denote
\begin{align*}
\widetilde{\boldsymbol{W}} = 
\begin{pmatrix}
\boldsymbol{W} & \boldsymbol{0}\\
\boldsymbol{0} & \boldsymbol{0}
\end{pmatrix}
\in \mathbb{R}^{(n+k) \times (n+k)}.
\end{align*}
The following theorem provides a simple sufficient condition for the equality \eqref{G2E} in terms of the commutativity of $\widetilde{\boldsymbol{W}}$ and $\boldsymbol{P}_Z$.
\begin{thm}\label{thm4}
Let $\boldsymbol{\Omega}$ be given by \eqref{omega} with $\rho \neq 0$.
If 
\begin{align}\label{new2}
\widetilde{\boldsymbol{W}}\boldsymbol{P}_Z = \boldsymbol{P}_Z\widetilde{\boldsymbol{W}},
\end{align}
then the equality \eqref{G2E} holds.
\end{thm}
In general, \eqref{new2} is not necessary for the equality \eqref{G2E} to hold.
However, if we impose the additional assumptions, the sufficient condition in Theorem~\ref{thm4} becomes necessary.
\begin{cor}\label{cor2}
Let $\boldsymbol{\Omega}$ be given by \eqref{omega}.
Suppose further that $\boldsymbol{W}$ is symmetric and $\rho \in \mathcal{P}$.
Then the equality \eqref{G2E} holds if and only if \eqref{new2}.
\end{cor}

\section{Numerical experiment}\label{sec:4}
In this section, we numerically illustrate Theorem~\ref{thm3}. 
Specifically, we compare cases in which the statement (iii) holds and does not hold, and examine the difference between two general ridge estimators over the admissible range of $\rho$.

Consider $n=16$ spatial units arranged on a directed $4\times4$ toroidal lattice. 
Let
\begin{align*}
\boldsymbol{S} = 
\begin{pmatrix}
0 & 1 & 0 & 0\\
0 & 0 & 1 & 0\\
0 & 0 & 0 & 1\\
1 & 0 & 0 & 0
\end{pmatrix}, \quad \boldsymbol{W} = \cfrac{1}{2} \ (\boldsymbol{I}_4 \otimes \boldsymbol{S} + \boldsymbol{S} \otimes \boldsymbol{I}_4).
\end{align*}
Then $\boldsymbol{W}$ is row-standardized and nonsymmetric.
Define
\begin{align*}
\boldsymbol{a}_0 = \cfrac{1}{2}
\begin{pmatrix}
1\\
1\\
1\\
1
\end{pmatrix}, \ \boldsymbol{a}_1 = \cfrac{1}{\sqrt{2}}
\begin{pmatrix}
1\\
0\\
-1\\
0
\end{pmatrix}, \ \boldsymbol{a}_2 = \cfrac{1}{\sqrt{2}}
\begin{pmatrix}
0\\
1\\
0\\
-1
\end{pmatrix}, \
\boldsymbol{a}_3 = \cfrac{1}{2} 
\begin{pmatrix}
1\\
-1\\
1\\
-1
\end{pmatrix}.
\end{align*}
The design matrix is given by
\begin{align*}
\boldsymbol{X} = (\boldsymbol{a}_0 \otimes \boldsymbol{a}_0, \boldsymbol{a}_1 \otimes \boldsymbol{a}_0, \boldsymbol{a}_2 \otimes \boldsymbol{a}_0, \boldsymbol{a}_3 \otimes \boldsymbol{a}_0, \boldsymbol{a}_0 \otimes \boldsymbol{a}_3) \in \mathbb{R}^{16 \times 5},
\end{align*}
which satisfies $\boldsymbol{X}^\top \boldsymbol{X} = \boldsymbol{I}_5$.
A direct calculation gives 
\begin{align*}
(\boldsymbol{W} + \boldsymbol{W}^\top)\boldsymbol{X} = \boldsymbol{X}\boldsymbol{G}_1 \quad  \an \quad \boldsymbol{W}\boldsymbol{W}^\top\boldsymbol{X} = \boldsymbol{X}\boldsymbol{G}_2,
\end{align*}
where 
\begin{align*}
\boldsymbol{G}_1 = \diag(2,1,1,0,0) \quad \an \quad \boldsymbol{G}_2 = \diag\left(1, \dfrac{1}{2}, \dfrac{1}{2},0,0\right).
\end{align*}
Here, $\diag(\cdot)$ denotes the diagonal matrix.
We compare two matrices $\boldsymbol{K}_1 = \diag(0,0,0,1,1)$ and $\boldsymbol{K}_2 = \boldsymbol{I}_5$.
Since $\boldsymbol{K}_1\boldsymbol{G}_1 = \boldsymbol{K}_1\boldsymbol{G}_2 = \boldsymbol{0}$, both conditions in Theorem~\ref{thm3} hold for $\boldsymbol{K}_1$.
In contrast, since it holds that
\begin{align*}
\boldsymbol{K}_2\boldsymbol{G}_1 = \boldsymbol{G}_1 \neq \boldsymbol{0} \quad \an \quad \boldsymbol{K}_2\boldsymbol{G}_2 = \boldsymbol{G}_2 \neq \boldsymbol{0},
\end{align*}
the statement (iii) is not satisfied for $\boldsymbol{K}_2$.

For SMA(1) process, the dispersion matrix is $\boldsymbol{\Omega}(\rho) = (\boldsymbol{I}_{16} + \rho\boldsymbol{W})(\boldsymbol{I}_{16} + \rho \boldsymbol{W}^\top)$. 
To compare two general ridge estimators, we calculate
\begin{align*}
D(\rho, \boldsymbol{K}_i) = \left\|\left(\boldsymbol{X}^\top \boldsymbol{\Omega}(\rho)^{-1}\boldsymbol{X} + \boldsymbol{K}_i\right)^{-1}\boldsymbol{X}^\top \boldsymbol{\Omega}(\rho)^{-1} - \left(\boldsymbol{X}^\top\boldsymbol{X} + \boldsymbol{K}_i\right)^{-1}\boldsymbol{X}^\top\right\|_{\mathrm{F}}, \quad i = 1,2
\end{align*}
for $\rho \in \mathcal{P}$.

\begin{figure}[t]
\centering
\caption{Difference between two general ridge estimators for
$\boldsymbol{K}_1=\operatorname{diag}(0,0,0,1,1)$ and
$\boldsymbol{K}_2=\boldsymbol{I}_5$}
\includegraphics[width=0.60\linewidth]
{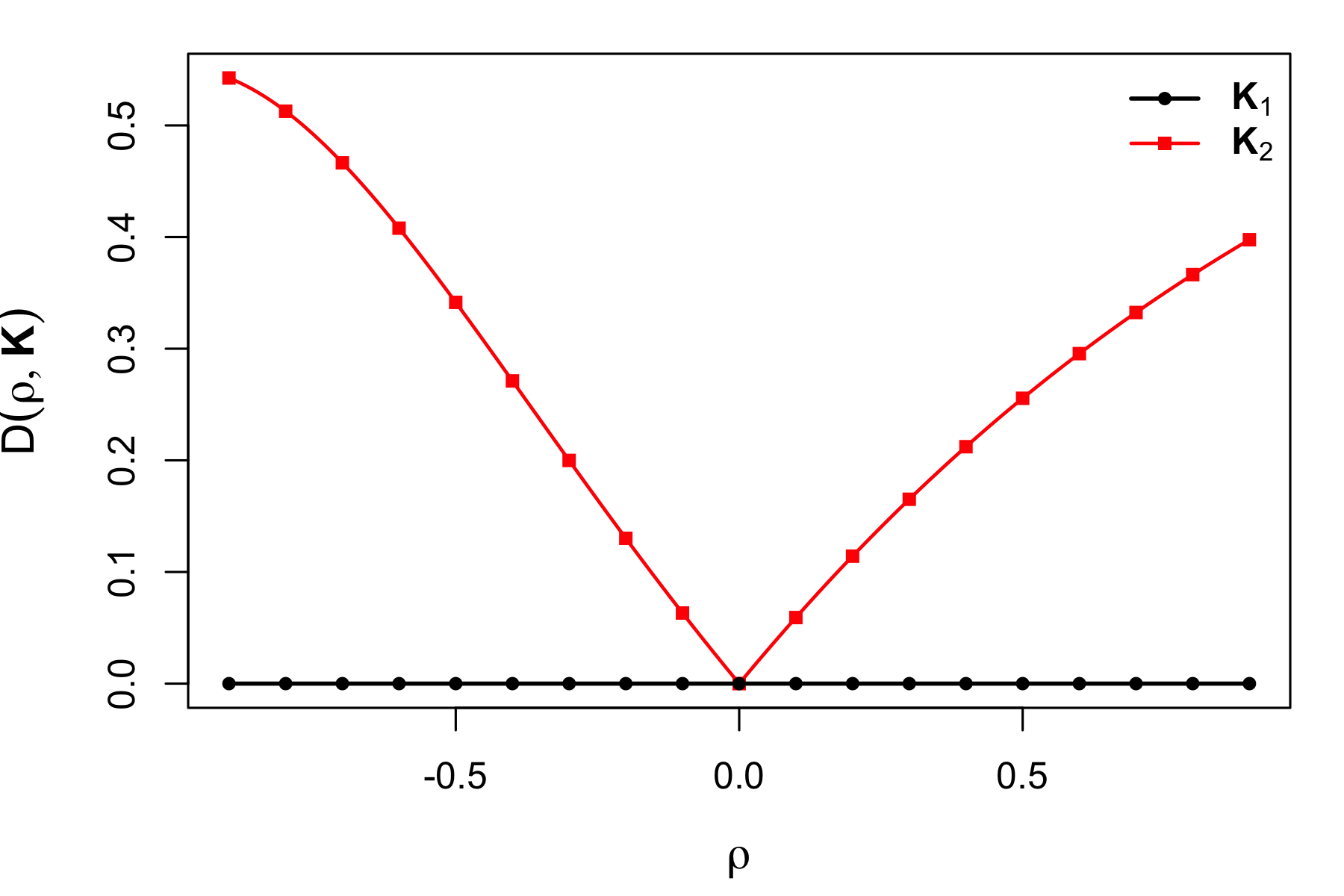}
\label{fig1}
\end{figure}
Figure~\ref{fig1} shows that $D(\rho, \boldsymbol{K}_1)$ remains at the level of numerical round-off error for all $\rho$, whereas $D(\rho, \boldsymbol{K}_2)$ is positive for $\rho \neq 0$.
These results illustrate the practical usefulness of the parameter-free condition in Theorem~\ref{thm3}.
 
\section{Concluding remarks}\label{sec:5}
In this paper, we investigated conditions under which two general ridge estimators coincide under spatial error models. 
First, in the general linear model, we derived a necessary and sufficient condition based on the commutativity of an extended dispersion matrix and an orthogonal projector, thereby extending the classical condition of \cite{RefZ67} to general ridge estimation. 
We then established algebraic conditions that do not depend on the unknown spatial correlation coefficient $\rho$ for the SAR(1) and SMA(1) processes. 
Additionally, we obtained a necessary and sufficient condition for the equality of two general ridge estimators for all $\rho \in \mathcal{P}$ and characterized all matrices $\boldsymbol{K}\in\MS^N(k)$ satisfying this condition. 
The numerical experiment showed that the two estimators coincide up to numerical round-off error when the proposed conditions are satisfied, whereas a clear difference arises when they are not satisfied. 
These results indicate that checking the proposed parameter-free conditions can simplify the two-step estimation procedure by eliminating the need to estimate $\rho$.

Several issues remain for future research. 
First, it would be of interest to extend the present results from SAR(1) and SMA(1) processes to higher-order spatial error processes.
Another important direction is to derive corresponding equality conditions for singular general linear models, that is, for the case where the design matrix is not full-rank or the dispersion matrix is singular.

\section*{Acknowledgements}
The author is also deeply grateful to his supervisor, Prof. Koji Tsukuda, for his insightful advice.

\appendix

\section*{Appendix}

\addcontentsline{toc}{section}{Appendix}

\section{Proofs}\label{app:proofs}

\subsection{Proof of Theorem~\ref{thm1}}

\begin{proof} 
It suffices to show that \eqref{Zcond} is equivalent to \eqref{Mcond2}.

First, we assume that $\boldsymbol{P}_Z \widetilde{\boldsymbol{\Omega}} = \widetilde{\boldsymbol{\Omega}}\boldsymbol{P}_Z$.
Postmultiplying by $\boldsymbol{Z}$ yields 
\begin{align*}
\boldsymbol{P}_Z \widetilde{\boldsymbol{\Omega}} \boldsymbol{Z} = \widetilde{\boldsymbol{\Omega}}\boldsymbol{Z}
\quad \Leftrightarrow \quad  \widetilde{\boldsymbol{\Omega}}\boldsymbol{Z} = \boldsymbol{Z}(\boldsymbol{Z}^\top \boldsymbol{Z})^{-1}\boldsymbol{Z}^\top\widetilde{\boldsymbol{\Omega}}\boldsymbol{Z},
\end{align*}
which implies that $\MC(\widetilde{\boldsymbol{\Omega}}\boldsymbol{Z}) \subseteq \MC(\boldsymbol{Z})$.
Since $\widetilde{\boldsymbol{\Omega}}$ is nonsingular, we have $\rk(\widetilde{\boldsymbol{\Omega}}\boldsymbol{Z}) = \rk(\boldsymbol{Z})$.
Therefore, we obtain $\MC(\widetilde{\boldsymbol{\Omega}}\boldsymbol{Z}) = \MC(\boldsymbol{Z})$.

Conversely, suppose that $\MC(\widetilde{\boldsymbol{\Omega}}\boldsymbol{Z}) = \MC(\boldsymbol{Z})$.
It follows from $\MC(\widetilde{\boldsymbol{\Omega}}\boldsymbol{Z}) \subseteq  \MC(\boldsymbol{P}_Z)$ that 
\begin{align}\label{1}
&(\boldsymbol{I}_{n+k} - \boldsymbol{P}_Z)\widetilde{\boldsymbol{\Omega}}\boldsymbol{P}_Z = \boldsymbol{0} \notag\\
\Leftrightarrow \quad & \widetilde{\boldsymbol{\Omega}}\boldsymbol{P}_Z = \boldsymbol{P}_Z\widetilde{\boldsymbol{\Omega}}\boldsymbol{P}_Z.
\end{align}
The transpose of \eqref{1} yields
\begin{align}\label{2}
\boldsymbol{P}_Z\widetilde{\boldsymbol{\Omega}} = \boldsymbol{P}_Z \widetilde{\boldsymbol{\Omega}}\boldsymbol{P}_Z.
\end{align}
From \eqref{1} and \eqref{2}, we have $\boldsymbol{P}_Z\widetilde{\boldsymbol{\Omega}} = \widetilde{\boldsymbol{\Omega}}\boldsymbol{P}_Z$.
This completes the proof.
\end{proof}

\subsection{Proof of Theorem~\ref{thm2}}
\begin{proof}
Under \eqref{new}, there exist matrices $\boldsymbol{G}_1, \boldsymbol{G}_2 \in \mathbb{R}^{k \times k}$ such that 
\begin{align*}
\begin{pmatrix}
\boldsymbol{WX}\\
\boldsymbol{0}
\end{pmatrix}
= 
\begin{pmatrix}
\boldsymbol{X}\\
\boldsymbol{K}
\end{pmatrix}\boldsymbol{G}_1 \quad \an \quad 
\begin{pmatrix}
\boldsymbol{W}^\top\boldsymbol{X}\\
\boldsymbol{0}
\end{pmatrix}
= 
\begin{pmatrix}
\boldsymbol{X}\\
\boldsymbol{K}
\end{pmatrix}\boldsymbol{G}_2.
\end{align*}
First, we consider the case where $\boldsymbol{\Omega} = (\boldsymbol{I}_n + \rho \boldsymbol{W})(\boldsymbol{I}_n + \rho \boldsymbol{W}^\top)$.
Then it holds that
\begin{align*}
\begin{pmatrix}
\boldsymbol{\Omega}\boldsymbol{X}\\
\boldsymbol{K}
\end{pmatrix}
= \begin{pmatrix}
\boldsymbol{X}\\
\boldsymbol{K}
\end{pmatrix}
(\boldsymbol{I}_k + \rho \boldsymbol{G}_1)(\boldsymbol{I}_k + \rho \boldsymbol{G}_2).
\end{align*}
Since the matrix on the left-hand side has full column rank $k$,  $(\boldsymbol{I}_k + \rho \boldsymbol{G}_1)(\boldsymbol{I}_k + \rho \boldsymbol{G}_2)$ is nonsingular.
Hence \eqref{Mcond} holds.

Next, we consider the case where $\boldsymbol{\Omega} = (\boldsymbol{I}_n - \rho \boldsymbol{W})^{-1}(\boldsymbol{I}_n - \rho \boldsymbol{W}^\top)^{-1}$.
Noting that \eqref{Mcond} is equivalent to
\begin{align*}
\MC\left(\begin{matrix}
\boldsymbol{\Omega}^{-1}\boldsymbol{X}\\
\boldsymbol{K}
\end{matrix}\right) = \MC\left(\begin{matrix}
\boldsymbol{X}\\
\boldsymbol{K}
\end{matrix}\right),
\end{align*}
we can demonstrate the statement in a similar manner.
This completes the proof.
\end{proof}

\subsection{Proof of Theorem~\ref{thm3}}
We introduce a technical lemma to prove Theorem~\ref{thm3}.
\begin{lem}\label{lem}
Let $\boldsymbol{\Omega}$ be written in the form
$\boldsymbol{\Omega} = (\boldsymbol{I}_n + \rho \boldsymbol{W})(\boldsymbol{I}_n + \rho \boldsymbol{W}^\top)$
with $\rho \neq 0$.
Then the condition \eqref{Mcond} is equivalent to
\begin{align}\label{column}
\MC\left(\begin{matrix}
(\boldsymbol{W} + \boldsymbol{W}^\top + \rho \boldsymbol{WW}^\top)\boldsymbol{X}\\
\boldsymbol{0}
\end{matrix}\right) \subseteq \MC\left(
\begin{matrix} 
\boldsymbol{X}\\
\boldsymbol{K}
\end{matrix}\right).
\end{align}
\end{lem}

\begin{proof}
Since $\boldsymbol{\Omega}\boldsymbol{X} = \boldsymbol{X} + \rho(\boldsymbol{W} + \boldsymbol{W}^\top + \rho\boldsymbol{W}\boldsymbol{W}^\top)\boldsymbol{X}$, we have
\begin{align*}
\begin{pmatrix}
\boldsymbol{\Omega}\boldsymbol{X}\\
\boldsymbol{K}
\end{pmatrix}
= 
\begin{pmatrix}
\boldsymbol{X}\\
\boldsymbol{K}
\end{pmatrix}
+ 
\rho
\begin{pmatrix}
(\boldsymbol{W} + \boldsymbol{W}^\top + \rho\boldsymbol{W}\boldsymbol{W}^\top)\boldsymbol{X}\\
\boldsymbol{0}
\end{pmatrix}.
\end{align*}
To see the necessity, suppose that \eqref{Mcond} holds.
Then there exists a nonsingular matrix $\boldsymbol{G} \in \mathbb{R}^{k \times k}$ such that
\begin{align*}
\begin{pmatrix}
\boldsymbol{\Omega}\boldsymbol{X}\\
\boldsymbol{K}
\end{pmatrix}
= 
\begin{pmatrix}
\boldsymbol{X}\\
\boldsymbol{K}
\end{pmatrix}\boldsymbol{G}.
\end{align*}
Therefore, we obtain
\begin{align*}
\begin{pmatrix}
(\boldsymbol{W} + \boldsymbol{W}^\top + \rho \boldsymbol{WW}^\top)\boldsymbol{X}\\
\boldsymbol{0}
\end{pmatrix} = 
\begin{pmatrix}
\boldsymbol{X}\\
\boldsymbol{K}
\end{pmatrix}
\rho^{-1}(\boldsymbol{G} - \boldsymbol{I}_k),
\end{align*}
which proves \eqref{column}.

Next, to see the sufficiency, suppose that \eqref{column} is satisfied.
Then there exists a matrix $\boldsymbol{G} \in \mathbb{R}^{k \times k}$ such that
\begin{align*}
\begin{pmatrix}
(\boldsymbol{W} + \boldsymbol{W}^\top + \rho \boldsymbol{WW}^\top)\boldsymbol{X}\\
\boldsymbol{0}
\end{pmatrix} = \begin{pmatrix}
\boldsymbol{X}\\
\boldsymbol{K}
\end{pmatrix}\boldsymbol{G}.
\end{align*}
It holds that
\begin{align*}
\begin{pmatrix}
\boldsymbol{\Omega}\boldsymbol{X}\\
\boldsymbol{K}
\end{pmatrix}
= 
\begin{pmatrix}
\boldsymbol{X}\\
\boldsymbol{K}
\end{pmatrix}
(\boldsymbol{I}_k + \rho \boldsymbol{G}).
\end{align*}
Since $\boldsymbol{I}_k + \rho \boldsymbol{G}$ is nonsingular, \eqref{Mcond} holds.
This completes the proof.
\end{proof}

From Lemma~\ref{lem}, we prove Theorem~\ref{thm3}.
\begin{proof}[Proof of Theorem~\ref{thm3}]
(i) $\Rightarrow$ (ii): This statement is trivial. \vspace{1ex}\\
(ii) $\Rightarrow$ (iii):
It follows from Lemma~\ref{lem} that the condition (ii) holds if and only if 
\begin{align*}
\MC\left(\begin{matrix}
(\boldsymbol{W}+\boldsymbol{W}^\top + \rho_1 \boldsymbol{W}\boldsymbol{W}^\top)\boldsymbol{X}\\
\boldsymbol{0}
\end{matrix}\right) \subseteq \MC\left(\begin{matrix}
\boldsymbol{X}\\
\boldsymbol{K}
\end{matrix}\right)
\end{align*}
and 
\begin{align*}
\MC\left(\begin{matrix}
(\boldsymbol{W}+\boldsymbol{W}^\top + \rho_2 \boldsymbol{W}\boldsymbol{W}^\top)\boldsymbol{X}\\
\boldsymbol{0}
\end{matrix}\right) \subseteq \MC\left(\begin{matrix}
\boldsymbol{X}\\
\boldsymbol{K}
\end{matrix}\right) 
\end{align*}
simultaneously hold for two different $\rho_1, \rho_2 \in \mathcal{P}$.
This means that there exist matrices $\boldsymbol{G}_1, \boldsymbol{G}_2 \in \mathbb{R}^{k \times k}$ such that
\begin{align}\label{eq1}
\begin{pmatrix}
(\boldsymbol{W}+\boldsymbol{W}^\top + \rho_1 \boldsymbol{W}\boldsymbol{W}^\top)\boldsymbol{X}\\
\boldsymbol{0}
\end{pmatrix} = 
\begin{pmatrix}
\boldsymbol{X}\\
\boldsymbol{K}
\end{pmatrix}
\boldsymbol{G}_1
\end{align}
and 
\begin{align}\label{eq2}
\begin{pmatrix}
(\boldsymbol{W}+\boldsymbol{W}^\top + \rho_2 \boldsymbol{W}\boldsymbol{W}^\top)\boldsymbol{X}\\
\boldsymbol{0}
\end{pmatrix} = 
\begin{pmatrix}
\boldsymbol{X}\\
\boldsymbol{K}
\end{pmatrix}
\boldsymbol{G}_2
\end{align}
Subtracting \eqref{eq2} from \eqref{eq1} yields
\begin{align}\label{eq3}
\MC\left(\begin{matrix}
\boldsymbol{W}\boldsymbol{W}^\top\boldsymbol{X}\\
\boldsymbol{0}
\end{matrix}\right) \subseteq \MC\left(\begin{matrix}
\boldsymbol{X}\\
\boldsymbol{K}
\end{matrix}\right)
\end{align} 
under $\rho_1 \neq \rho_2$.
From \eqref{eq3}, there exists a matrix $\boldsymbol{F} \in \mathbb{R}^{k \times k}$ such that 
\begin{align*}
\begin{pmatrix}
\boldsymbol{W}\boldsymbol{W}^\top\boldsymbol{X}\\
\boldsymbol{0}
\end{pmatrix}
= 
\begin{pmatrix}
\boldsymbol{X}\\
\boldsymbol{K}
\end{pmatrix}
\boldsymbol{F}
\end{align*}
Then it holds that 
\begin{align*}
\begin{pmatrix}
(\boldsymbol{W} + \boldsymbol{W}^\top)\boldsymbol{X}\\
\boldsymbol{0}
\end{pmatrix}
= 
\begin{pmatrix}
\boldsymbol{X}\\
\boldsymbol{K}
\end{pmatrix}
(\boldsymbol{G}_1 - \rho_1\boldsymbol{F})
\end{align*}
for $\rho_1 \in \mathcal{P}$.
Therefore, we have 
\begin{align*}
\MC\left(\begin{matrix}
(\boldsymbol{W}+\boldsymbol{W}^\top)\boldsymbol{X}\\
\boldsymbol{0}
\end{matrix}\right) \subseteq \MC\left(\begin{matrix}
\boldsymbol{X}\\
\boldsymbol{K}
\end{matrix}\right).
\end{align*}
(iii) $\Rightarrow$ (i): Suppose that there exist matrices $\boldsymbol{F}_1, \boldsymbol{F}_2 \in \mathbb{R}^{k \times k}$ such that 
\begin{align*}
\begin{pmatrix}
(\boldsymbol{W} + \boldsymbol{W}^\top)\boldsymbol{X}\\
\boldsymbol{0}
\end{pmatrix}
= 
\begin{pmatrix}
\boldsymbol{X}\\
\boldsymbol{K}
\end{pmatrix}
\boldsymbol{F}_1
\quad \an \quad \begin{pmatrix}
\boldsymbol{W}\boldsymbol{W}^\top\boldsymbol{X}\\
\boldsymbol{0}
\end{pmatrix}
= 
\begin{pmatrix}
\boldsymbol{X}\\
\boldsymbol{K}
\end{pmatrix}
\boldsymbol{F}_2.
\end{align*}
Then it holds that 
\begin{align*}
\begin{pmatrix}
(\boldsymbol{W}+\boldsymbol{W}^\top + \rho \boldsymbol{W}\boldsymbol{W}^\top)\boldsymbol{X}\\
\boldsymbol{0}
\end{pmatrix}
=
\begin{pmatrix}
\boldsymbol{X}\\
\boldsymbol{K}
\end{pmatrix}
(\boldsymbol{F}_1 + \rho \boldsymbol{F}_2)
\end{align*}
which implies that \eqref{column} is satisfied.
By Lemma~\ref{lem}, the statement (i) is true.
This completes the proof.
\end{proof}

\subsection{Proof of Corollary~\ref{cor1}}
\begin{proof}
Under the assumptions, the statement (iii) in Theorem~\ref{thm3} holds if and only if  $\boldsymbol{K}\boldsymbol{R}_1 = \boldsymbol{K}\boldsymbol{R}_2 = \boldsymbol{0}$, or equivalently, $\boldsymbol{KR} = \boldsymbol{0}$.
See also Remark~\ref{rem33}-(b).
Since $\boldsymbol{K}$ is symmetric, $\boldsymbol{KR} = \boldsymbol{0}$ is equivalent to
$\MC(\boldsymbol{K}) \subseteq \MN(\boldsymbol{R}^\top) = \MC(\boldsymbol{Q})$.
Therefore, we have
\begin{align*}
\boldsymbol{K} = \boldsymbol{Q}\boldsymbol{Q}^\top\boldsymbol{K}\boldsymbol{Q}\boldsymbol{Q}^\top = \boldsymbol{Q}\boldsymbol{H}\boldsymbol{Q}^\top,
\end{align*}
where $\boldsymbol{H} = \boldsymbol{Q}^\top\boldsymbol{K}\boldsymbol{Q} \in \MS^N(q)$.
The converse follows from $\boldsymbol{Q}^\top \boldsymbol{R} = \boldsymbol{0}$.
This completes the proof.
\end{proof}

\subsection{Proof of Theorem~\ref{thm4}}
We prove the result only for SMA(1) process, where $\boldsymbol{\Omega} = (\boldsymbol{I}_n + \rho \boldsymbol{W})(\boldsymbol{I}_n + \rho \boldsymbol{W}^\top)$.
As in the proof of Theorem~\ref{thm2}, the result for SAR(1) process can be established analogously.
\begin{proof}
Note that $\widetilde{\boldsymbol{\Omega}}$ can be rewritten as
\begin{align*}
\widetilde{\boldsymbol{\Omega}}  = (\boldsymbol{I}_{n+k} + \rho \widetilde{\boldsymbol{W}})(\boldsymbol{I}_{n+k} + \rho \widetilde{\boldsymbol{W}}^\top)  = \boldsymbol{I}_{n+k} + \rho (\widetilde{\boldsymbol{W}} + \widetilde{\boldsymbol{W}}^\top) + \rho^2 \widetilde{\boldsymbol{W}}\widetilde{\boldsymbol{W}}^\top.
\end{align*}
Since $\widetilde{\boldsymbol{W}}\boldsymbol{P}_Z = \boldsymbol{P}_Z \widetilde{\boldsymbol{W}}$, it holds that
\begin{align*}
(\widetilde{\boldsymbol{W}} + \widetilde{\boldsymbol{W}}^\top)\boldsymbol{P}_Z = \boldsymbol{P}_Z(\widetilde{\boldsymbol{W}} + \widetilde{\boldsymbol{W}}^\top), \quad \widetilde{\boldsymbol{W}}\widetilde{\boldsymbol{W}}^\top \boldsymbol{P}_Z = \boldsymbol{P}_Z\widetilde{\boldsymbol{W}}\widetilde{\boldsymbol{W}}^\top.
\end{align*}
This implies that \eqref{Zcond} holds, and thus the equality \eqref{G2E} holds from Theorem~\ref{thm1}.
This completes the proof.
\end{proof}

\subsection{Proof of Corollary~\ref{cor2}}
We prove the result for the SMA(1) process, where $\boldsymbol{\Omega} = (\boldsymbol{I}_n + \rho \boldsymbol{W})(\boldsymbol{I}_n + \rho \boldsymbol{W}^\top)$.
\begin{proof}
The sufficiency follows immediately from Theorem~\ref{thm4}. We therefore prove the necessity.
Suppose that the equality \eqref{G2E} holds. 
It follows from Theorem~\ref{thm1} that we have \eqref{Zcond},
where $\widetilde{\boldsymbol{\Omega}} = (\boldsymbol{I}_{n+k} + \rho \widetilde{\boldsymbol{W}})^2$.
Because $\boldsymbol{W}$ is symmetric and row standardized, all the eigenvalues of $\widetilde{\boldsymbol{W}}$ lie in the interval $[-1,1]$.
Hence, $\rho \in \mathcal{P}$ implies that $\boldsymbol{A} = \boldsymbol{I}_{n+k} + \rho\widetilde{\boldsymbol{W}}$ is positive definite.
Therefore, $\boldsymbol{A}$ is the unique positive definite square root of $\widetilde{\boldsymbol{\Omega}}$.
Since $\widetilde{\boldsymbol{\Omega}}$ and $\boldsymbol{P}_Z$ commute, they are simultaneously orthogonally diagonalizable. 
 Thus, there exist an orthogonal matrix $\boldsymbol{U}$ and diagonal matrices $\boldsymbol{\Lambda}$ and $\boldsymbol{D}$ such that $\widetilde{\boldsymbol{\Omega}} = \boldsymbol{U}\boldsymbol{\Lambda}\boldsymbol{U}^\top$ and $\boldsymbol{P}_Z = \boldsymbol{U}\boldsymbol{D}\boldsymbol{U}^\top.$
Then $\boldsymbol{A}$ can be expressed as $\boldsymbol{A} = \boldsymbol{U}\boldsymbol{\Lambda}^{1/2}\boldsymbol{U}^\top$, where $\boldsymbol{\Lambda}^{1/2}$ denotes the diagonal matrix whose diagonal entries are the positive square roots of the corresponding diagonal entries of $\boldsymbol{\Lambda}$.
Since $\boldsymbol{\Lambda}^{1/2}$ and $\boldsymbol{D}$ are diagonal, it holds that $\boldsymbol{A}\boldsymbol{P}_Z = \boldsymbol{P}_Z \boldsymbol{A}$.
Under $\rho \neq 0$, it follows that
\begin{align*}
(\boldsymbol{I}_{n+k} + \rho\widetilde{\boldsymbol{W}})\boldsymbol{P}_Z = \boldsymbol{P}_Z(\boldsymbol{I}_{n+k} + \rho\widetilde{\boldsymbol{W}}) \quad \Leftrightarrow \quad \widetilde{\boldsymbol{W}}\boldsymbol{P}_Z = \boldsymbol{P}_Z\widetilde{\boldsymbol{W}}.
\end{align*}
Thus, \eqref{new2} holds. 
This completes the proof.
\end{proof}
\end{document}